\documentclass[11pt, reqno]{amsart}   	
		
\usepackage{accents}	
\usepackage{adjustbox}
\usepackage{amssymb, latexsym}
\usepackage{amsmath}
\usepackage[toc,page]{appendix}
\usepackage{braket}
\usepackage{breqn}
\usepackage{caption}
\usepackage{color}
\usepackage{comment}
\usepackage{dsfont}
\usepackage{esint}
\usepackage{enumitem}
\usepackage[T1]{fontenc}
\usepackage{geometry}         
\usepackage{graphicx}
\usepackage{lipsum}
\usepackage{hyperref}
\usepackage{latexsym}
\usepackage{mathrsfs}
\usepackage{mathtools}
\usepackage{subcaption}
\usepackage{float}
\restylefloat{table}
\usepackage{listings}
\usepackage{wasysym}
\usepackage[dvipsnames]{xcolor} 

\usepackage{times}
\usepackage{url} 
\usepackage{tikz}
\usetikzlibrary{cd}
\usetikzlibrary{positioning}

\theoremstyle{plain}
\numberwithin{equation}{section}
\newtheorem{thm}{Theorem}[section]
\newtheorem{theorem}[thm]{Theorem}
\newtheorem{lemma}[thm]{Lemma}

\newtheorem{assumption}[thm]{Assumption}

\newtheorem{corollary}[thm]{Corollary}

\newtheorem{example}[thm]{Example}

\newtheorem{definition}[thm]{Definition}

\newtheorem{proposition}[thm]{Proposition}

\newtheorem{remark}[thm]{Remark}
\newtheorem{problem}[thm]{Problem}

\def\de{\delta}

\def\ep{\epsilon}

\def\Ga{\Gamma}

\def\la{\lambda}
\def\La{\Lambda}

\def\Om{\Omega}
\def\pa{\partial}

\def\vphi{\varphi}

\def\N{\mathbb{N}}

\def\R{\mathbb{R}}

\def\ex{\exists}

\def\grad{\nabla}
\def\lang{\langle}

\def\rang{\rangle}

\def\Bor{\mathcal{B}}

\def\Lin{\mathcal{L}}

\def\argmin{\text{argmin}}

\def\des{\text{des}}

\def\dist{\text{dist}}

\def\liminf{\text{liminf}}
\def\limsup{\text{limsup}}

\def\loc{\text{loc}}

\def\supp{\text{supp}}

\newcommand\beq{\begin{equation}}
\newcommand{\bburl}[1]{\textcolor{blue}{\url{#1}}}
\newcommand\eeq{\end{equation}}
\newcommand\bea{\begin{eqnarray}}
\newcommand\eea{\end{xq}}
\newcommand\bi{\begin{itemize}}
\newcommand\ei{\end{itemize}}
\newcommand\ben{\begin{enumerate}}
\newcommand\een{\end{enumerate}}

\everymath{\displaystyle}

\title{On convergence of solutions to nonlocal optimal control problems with quasi-minimization constraints}
\author{Joshua M. Siktar}
\date{\today}							

\begin{document}
\maketitle


\begin{abstract}
This paper studies a new class of nonlocal optimal control problems where the constraints involve approximations of minimizers of quasiconvex energies. These problems are parameterized by a fractional parameter $s \in (0, 1)$ and a horizon parameter $\de > 0$, and the energy density depends on a nonlocal fractional gradient. Here, the constraint consists of finding quasi-minimizers with respect to the quasiconvex energy. Despite the fact that the energies of interest may not have unique minimizers, we may prove the existence of solutions to this class of control problems. 

The constraint in \cite{cueto2026localization} was finding global minimizers of the energy, and this problem's main limitation was an inability to prove convergence of solutions for the nonlocal control problems to those of a corresponding local, PDE-constrained optimal control problem. While this issue arises from the lack of uniqueness of minimizers for the constraining energy, we get stronger convergence results with our new choice of constraints. Namely, we obtain convergence of minimizers for nonlocal optimal control problems that have a general cost functional depending on the nonlocal gradient.

\end{abstract}


\section{Introduction}\label{Sec: intro}




This paper considers ill-posed, nonlocal optimal control problems with a quasi-minimization constraint. For $\ep > 0$, an \textbf{$\ep$-quasi-minimizer} of a functional $\mathcal{G}$ is a function $u$ such that $\mathcal{G}(u) \leq \inf \mathcal{G}(\cdot) + \ep$, and this is reflected in the following problem statement:
        	\begin{equation}\label{Eq: genericQuasiMins}
		\begin{cases}
			\min\{\mathcal{F}(u, g) \ | \  (u, g) \in H^{s, p, \de}_0(\Om_{-\de}; \R^n) \times Z_{\text{ad}}\} \\
			\mathcal{W}^{\de, s}_g(u) - \argmin_{H^{s, p, \de}_0(\Om_{-\de}; \R^n)}\mathcal{W}^{\de, s}_g(\cdot) \leq \ep.
		\end{cases}
	\end{equation}

    Here $Z_{\text{ad}}$ is a class of input (control) variables, $H^{s, p, \de}_0$ is the energy space for the nonlocal energy density $\mathcal{W}^{\de, s}_g(\cdot)$, and $\mathcal{F}(\cdot, \cdot)$ is a cost functional; exact assumptions will be provided later. It is well-known that all critical points, including local minima, of appropriately defined quasiconvex energies  correspond to equilibrium solutions of Cauchy's equations of motion (see \cite{cauchy1828exercices, green1967formulation} for classical results and \cite{bellido2020bond} for connections to nonlocal hyperelasticity). 


We prove existence of solutions to problems of the form \eqref{Eq: genericQuasiMins}, and then consider convergence results of solutions to corresponding local, optimal control problems whose constraints depend on the classical gradient. This paper builds upon \cite{cueto2026localization}: by using quasi-minimization (instead of global minimization) constraints, we are now able to prove convergence of solutions for the family of nonlocal control problems. A more general cost functional depending on a nonlocal gradient is also utilized. Moreover, the main convergence results establish a novel analytic technique for proving convergence of solutions to parameterized optimal control problems beyond when nonlocalities are present, that involves leveraging quasi-minimization to carefully construct recovery sequences. We choose to study convergence through the lens of $\Ga$-convergence for the families of nonlocal energies.

 Since optimal control problems with constraints involving quasi-minimization are a rarity in the literature, we explain the relevance of these problems on two fronts. First, integral functionals arising in the study of relaxation theory (see  \cite[Theorem 2.16]{fonseca1993relaxation} for an example posed on $BV$-spaces, and \cite{buttazzo1982gamma, dal2012introduction, Rin} for an overview of relaxation theory) promote quasi-minimization by considering arbitrary sequences of functions that approximate a minimizer. This gives a method to feasibly study optimal control problems where the energy minimization constraint involves an energy without a classical minimizer, such as one containing a double-well potential. Second, from a computational standpoint, when we seek to minimize an energy functional we usually prescribe a tolerance threshold; in the case of finding quasi-minimizers we may view that tolerance as being accounted for in the analytical formulation of the problem. In fact, quasi-minimization has been utilized to formulate convergence of minima results in the setting of Discontinuous Galerkin methods for plate-bending models in \cite{bonito2023numerical, bonito2021dg}.

The problem  \eqref{Eq: genericQuasiMins} is posed in a function space framework involving nonlocal gradients; in other words, these function spaces admit functions possessing singularities and account for long-range interactions in physical and biological systems. The generic nonlocal gradient operator often takes the form
 \begin{equation}\label{Eq: genericNLGrad}
		\mathscr{G}_{\rho}u(x) \ := \ \int_{\Om}\rho(x - y)\frac{u(x) - u(y)}{|x - y|}  \frac{x - y}{|x - y|}dy,
	\end{equation}
	where $\rho \in L^1(\Om)$ is a kernel function containing a singularity at the origin. In this paper we assume $\Om \subset \R^n$ is bounded to more effectively model problems in solid mechanics. In fact, two of the most prominent classes of nonlocal models are designed as physically accurate representations of diffusion and solid mechanics, and the common characteristic is the utilization of integral operators in place of differential operators. The solid mechanics model of peridynamics (pioneered in \cite{silling2000reformulation, silling2007peridynamic, silling2005peridynamic}) uses $\de$ to measure the radius of interaction forces between particles.  An alternative taken in some papers is to define the nonlocal function spaces as the energy space of a double integral rather than the $L^p$ space of a nonlocal gradient of the form \eqref{Eq: genericNLGrad}; this is the approach used in \cite{bellido2014existence, mengesha2012nonlocal, mengesha2014bond, scott2019fractional}, among other places. 
    
    Meanwhile, in this paper we also utilize the parameter $s \in (0, 1)$ to represent a fractional order of differentiability, which arises in the Riesz $s$-fractional gradient of the form
    \begin{equation}\label{Eq: RieszKernel}
        \rho(\xi) \ := \ \frac{c_{n, s}}{|\xi|^{n - 1 + s}}.
    \end{equation}
    Now we may define the nonlocal gradient
    	\begin{equation}\label{Eq: introGrad}
		D^s_{\de}u(x) \ := \ c_{n, s}\int_{B(x, \de)}\frac{u(x) - u(y)}{|x - y|}  \frac{x - y}{|x - y|}\frac{w_{\de}(x - y)}{|x - y|^{n - 1 + s}}dy,
	\end{equation}
     where $\de > 0$ is a radius of interaction (known as the \textbf{horizon parameter}), and $w_{\de}: \R^n \rightarrow [0, \infty)$ is a radial, truncated, non-negative kernel whose exact properties will be specified in the notation section.  These spaces already have a well-established functional framework developed across several papers, including a notion of Fundamental Theorem of Calculus, Poincaré inequalities, and embeddings \cite{bellido2025nonlocal}; and nonlocal Green Theorems and Helmholtz decompositions \cite{bellido2024green}. Moreover, it is shown in \cite{kreisbeck2024non} that the kernel of \eqref{Eq: introGrad} is nontrivial, and this leads to this space being used to pose nonlocal boundary-value problems with Neumann boundary conditions. 
     
 Optimal control problems, including these nonlocal ones, follow a general blueprint: we consider a control input function into a constraint, and then output a state function that satisfies the constraint. Then, we minimize (or maximize) a cost functional $\mathcal{F}(\cdot, \cdot)$ over pairs of controls and states that satisfy the constraint (see \cite{hinze2008optimization, troltzsch2010optimal} for theory of PDE-constrained control problems), but unlike recent works \cite{cueto2026localization, d2014optimal, mengesha2023control, munoz2022local} we allow the cost functional to explicitly depend on the nonlocal gradient \eqref{Eq: introGrad}. Typically the constraint is an equation or system that has a unique state solution for any admissible control, but the present paper deviates from this by considering constraints that may have multiple solutions, because the underlying energy to be minimized is quasiconvex or even polyconvex (this difficulty also arises in \cite{cueto2026localization, siktar2024superlinear}). We note that a quasiconvex energy density is one that admits minimal energy with an affine deformation as compared to internally distorted deformations, while polyconvexity additionally allows adherence to the physical principle that infinite compressions cost infinite energy \cite{dacorogna2007direct, Rin}. Just as importantly, these are weaker notions of convexity more appropriate for the study of systems in nonlinear elasticity \cite{ball1976convexity}, and existence of minimizers for a quasi or polyconvex energy can still be guaranteed by the classical Morrey's Theorem \cite{acerbi1984semicontinuity, marcellini1985approximation, meyers1965quasi, morrey1952quasi}, but uniqueness is not guaranteed. Related ill-posedness issues also arise in the study of inverse and shape optimization problems; see \cite{allaire2001some, cherkaev1997topics} for discussion of these areas.

 With that in mind, one of the analytical motivations of these works is to investigate how the ill-posedness of the constraint influences existence and uniqueness theory for the control problems, and also convergence of solutions to those of local, PDE-constrained control problems.  
 We finally note that in all of these works, the asymptotic analysis in a varying nonlocal parameter is treated through $\Ga$-convergence (see \cite{buttazzo1982gamma, dal2012introduction} for the classical theory, and \cite{bellido2020bond, cueto2025gamma, cueto2026localization,  mengesha2023control, mengesha2025design, munoz2022local, schonberger2024nonlocal}  for the application of $\Ga$-convergence theory to nonlocal control problems); in this paper we consider the limits $s \rightarrow 1^-$ and $\de \rightarrow 0^+$, though \cite{cueto2025gamma} also considers the diverging horizon case ($\de \rightarrow +\infty$) outside the context of control problems. 

Now we outline the remaining sections of the paper. Section \ref{Sec: prelim} discusses notation, preliminary results, and precise problem statements. Next, Sections \ref{Sec: ExistenceQuasi} describes the existence theory for the quasi-minimizers control problem. Then, Section \ref{Sec: quasiMinimizersAsymptotics} details asymptotic results for the $\ep$-quasi-minimizers problem of the form \eqref{Eq: genericQuasiMins}, respectively.  Finally, Section \ref{Sec: Conclusion} gives concluding remarks and future problems.



\section{Preliminaries}\label{Sec: prelim}

\subsection{Notation}\label{Subsec: Notation}

We let $\Om \subset \R^n$ be an open, bounded domain, and let $B(x, r)$ be the open ball centered at $x \in \R^n$ with radius $r$. For $\de > 0$ fixed, we define $\Om_{\de} := \Om \cup \{x \in \R^n \ | \ \dist(x, \pa \Om) < \de\}$, and the nonlocal boundary as the set $\Om_{\de}\setminus \Om$. In addition, we define $\Om_{-\de} := \{x \in \Om: \dist(x, \pa \Om) > \de\}$. We also assume that $\Lin^n(\pa \Om_{-\de}) = 0$, where $\Lin^n(\cdot)$ denotes the Lebesgue measure on $\R^n$; meanwhile, $\Bor^n$ denotes the Borel measure on $\R^n$.

For any $p \in (1, \infty)$, $p'$ denotes the conjugate exponent of $p$, i.e., $p' = \frac{p - 1}{p}$. The exponent $p$ being fixed, we denote by $\lang \cdot, \cdot\rang$ the duality pairing $\lang \cdot, \cdot\rang_{L^{p'}(\Om; \R^n), L^p(\Om; \R^n)}$. Meanwhile, $\cdot$ denotes the Euclidean dot product on $\R^n$, and $\otimes$ denotes the standard tensor (outer) product on $\R^n$. Finally, $\argmin$ denotes a solution to a given minimization problem.

Meanwhile, a prototypical form of our nonlocal-fractional gradient, a smoothly truncated version of \eqref{Eq: RieszKernel}, is
\begin{equation}\label{Eq: kernelRhosDe}
		\rho^s_{\de}(x) \ = \ \frac{c_{n,s}}{ |x|^{n-1+s}}w_\de(x) \quad \text{ with } \quad c_{n,s} \:= \ \frac{n+s-1}{\gamma(1-s)},
	\end{equation}
    where $w_{\de}(\cdot)$ satisfies the following assumptions.

\begin{assumption}[Kernel assumptions]\label{Assump: Kernels}
		The family of truncated kernels $\{w_{\de}\}_{\de > 0}$ satisfies:
		\begin{enumerate}
			\item $w_\delta:\R^n\to[0,\infty)$ is radial; i.e. $w_\de(x)=\bar{w}_\de(|x|)$ for some non-negative $\bar{w}_\de\in C_c^\infty([0,\infty))$, with $\supp(\bar{w}_\de)\subset [0,\delta)$.
			\item There is a constant $0<b_0<1$ such that $\bar{w}|_{[0,b_0\de]}=a_0$, where $a_0 := \max_{r\ge 0}\bar{w}_\de(r)$.
			\item
			$\bar{w}_\de (r_1) \geq \bar{w}_\de (r_2)$ whenever $r_1 \leq r_2$.	
   \end{enumerate}
	\end{assumption}
    Under Assumption \ref{Assump: Kernels}, the kernel $\rho^s_{\de}$ defined in  \eqref{Eq: kernelRhosDe} belongs to $L^1 (\R^n)$. These assumptions also let us precisely define our nonlocal gradient (see \cite[Definition 3.1]{bellido2023non}).

    	\begin{definition}[Nonlocal gradient]\label{Def: NLGrad}
		Let $0\leq s < 1$ and $\de>0$. For $u\in C^\infty_c(\R^n; \R^n)$ the nonlocal gradient $D^s_\de u$ is defined as
			\begin{equation}\label{Eq: NLGrad}
				D^s_\de u(x) \ := \ \int_{B(x,\de)} \frac{u(x)-u(y)}{|x-y|} \otimes \frac{x-y}{|x-y|}\rho^s_{\de}(x - y) \, dy.
			\end{equation}
	\end{definition}

Despite \eqref{Eq: NLGrad} initially being defined for smooth functions, the natural energy space is is the closure of smooth compactly supported functions with respect to the given norm.

	\begin{definition}\label{Def: HspdeNorm}
		Let $1 \leq p <\infty$, $s \in (0, 1)$, $\de \in (0, \infty)$. Then $H^{s,p,\de}(\Om; \R^n)$ is defined as the closure of $C^{\infty}_c(\R^n; \R^n)$ under the norm
		\begin{equation}
			\|u\|_{H^{s,p,\de}(\Om; \R^n)}\ := \ \|u\|_{L^p(\Om_{\de}; \R^n)} + \|D^s_\de u\|_{L^p(\Om; \R^{n\times n})}.
		\end{equation}
        		In addition, the set of functions in $H^{s,p,\de}(\Om; \R^n)$ that vanishes outside of  $\Om_{-\de}$ is denoted as $H^{s, p, \de}_0(\Om_{-\de}; \R^n)$. 
	\end{definition}
    
\subsection{Functional analytic framework}\label{Subsec: analytic}

This subsection will catalog preliminary results needed on the function spaces described in Subsection \ref{Subsec: Notation}. We start with basic properties of these spaces, contained in \cite[Propositions 2.5-2.6]{bellido2023minimizers} and \cite[Corollary 1]{cueto2023variational}:
\begin{itemize}
    \item For $p \in [1, \infty)$, $(H^{s, p, \de}(\Om; \R^n), \|\cdot\|_{H^{s, p, \de}(\Om; \R^n)})$ is a separable Banach space, and is Hilbert when $p = 2$.
		\item If $p \in (1, \infty)$, then $(H^{s, p, \de}(\Om; \R^n), \|\cdot\|_{H^{s, p, \de}(\Om; \R^n)})$ is also reflexive.
		\item Whenever $1 \leq q \leq p < \infty$, we have the continuous embedding $H^{s, p, \de}(\Om; \R^n) \subset H^{s, q, \de}(\Om; \R^n)$.
		\item If $0 \leq s_1 \leq s_2 \leq 1$ and $p \in (1, \infty)$, then $H^{s_2, p, \de}_0(\Om; \R^n) \subset H^{s_1, p, \de}_0(\Om; \R^n)$ is a continuous embedding with an embedding constant depending only on $n$, $\de$, and $p$.
\end{itemize}

From here on we only consider $p \in (1, \infty)$.

Now we mention a nonlocal version of Poincaré's Inequality (first proven in \cite[Theorem 4]{cueto2023variational}).
\begin{proposition}[Nonlocal Poincaré inequality]\label{Prop: NLPoincare}
		Let $s \in [0, 1]$,  $\delta > 0$, $p \in (1, \infty)$, $\Om \subset \R^n$ be open and bounded. Then there exists a constant $C > 0$ such that for all $u \in H^{s, p, \de}_0(\Om; \R^n)$, we have
		\begin{equation}\label{Eq: NLPoincare}
			\|u\|_{L^p(\Om; \R^n)} \ \leq \ C\|D^s_{\de}u\|_{L^p(\Om; \R^{n\times n})}.
		\end{equation}
        \end{proposition}
This version will be sufficient for the existence theory in Section \ref{Sec: ExistenceQuasi}. However, $s$ and $\de$ will both vary at different points in Section \ref{Sec: quasiMinimizersAsymptotics}, so in that section we mention versions of \eqref{Eq: NLPoincare} with constants independent of $s$ or $\de$, respectively.

Now we mention the requisite compact embedding. This appears in \cite[Theorem 7.3]{bellido2023non} (see also \cite[Theorem 2.9]{bellido2023minimizers}), but we only mention the $p \in (1, \infty)$ case.
    	\begin{proposition}[Compact embedding]\label{Prop: CompactEmbedding}
		Let $s \in (0, 1)$, $p \in (1, \infty)$, and $\de > 0$ are fixed. Suppose $\{u_j\}^{\infty}_{j = 1} \subset H^{s, p, \de}_0(\Om_{-\de}; \R^n)$ is a sequence such that $u_j \rightharpoonup u$ weakly in $H^{s, p, \de}(\Om; \R^n)$. Then $u_j \rightarrow u$ strongly in $L^q(\Om; \R^n)$, where $q$ satisfies the following:
		\begin{itemize}
			\item $q \in [1, p^*_s)$ if $sp < n$;
			\item $q \in [1, \infty)$ if $sp = n$;
			\item $q \in [1, \infty]$ if $sp > n$,
		\end{itemize}
		where $p^*_s := \frac{np}{n - sp}$. In addition, we have that $u \in H^{s, p, \de}_0(\Om_{-\de}; \R^n)$.
	\end{proposition}

    Finally, for our asymptotic analysis, we will need variants for when $s \rightarrow 1^-$ or $\de \rightarrow 0^+$.

    \begin{proposition}[Compactness as $s \rightarrow 1^-$]\label{Prop: CompactVarys->1}
        Suppose that $\de > 0$, $p \in (1, \infty)$ are fixed and $\{u_s\}_{s < 1}$ is a family of functions in $H^{s, p, \de}_0(\Om; \R^n)$ such that $\liminf_{s \rightarrow 1^-}\|D^s_{\de}u_s\|_{L^p(\Om; \R^{n \times n})} < \infty$. Then there exists a (not relabeled) sub-sequence and a $u \in W^{1, p}_0(\Om; \R^n)$ such that $D^s_{\de}u_s \rightharpoonup \grad u$ weakly in $L^p(\Om; \R^{n \times n})$ and $u_s \rightarrow u$ strongly in $L^p(\Om; \R^n)$. 
    \end{proposition}

    \begin{proposition}[Compactness as $\de \rightarrow 0^+$]\label{Prop: CompactVaryde-0}
         Suppose that $s \in (0, 1), p \in (1, \infty)$ are fixed and $\{u_{\de}\}_{\de > 0}$ is a family of functions in $H^{s, p, \de}_0(\Om; \R^n)$ such that $\liminf_{\de \rightarrow 0^+}\|D^s_{\de} u_{\de}\|_{L^p(\Om_{\de}; \R^{n \times n})} < \infty$. Then there exists a (not relabeled) sub-sequence and a $u \in W^{1, p}_0(\Om; \R^n)$ such that $D^s_{\de}u_{\de} \rightharpoonup \grad u$ weakly in $L^p(\Om; \R^{n \times n})$ and $u_{\de} \rightarrow u$ strongly in $L^p(\Om; \R^n)$. 
    \end{proposition}
    These results are \cite[Lemma 9]{cueto2023variational} and \cite[Lemma 3.6]{cueto2025gamma} (see also \cite[Lemma 6.3.6]{schonberger2024nonlocal}), respectively. They also apply in situations with nonzero nonlocal boundary data.


\subsection{Notions of convexity, energies, and minimality}\label{Subsec: energy}

Now recall, as discussed in the introduction, the more standard notions of convexity in vector-valued problems are polyconvexity and quasiconvexity. We define these (classical) notions here.

\begin{definition}[Polyconvex energy]\label{Def: pConvexEnergy}
		Let $F\in \R^{n\times n}$ and $\mu(F)$ denote the collection of all its minors. Then, a function $W: \Om \times \R^n \times \R^{n \times n} \rightarrow \R$ is said to be \textbf{polyconvex} if there exists a convex function $\bar{W}$ such that
		\begin{equation}\label{eq: pConvexEnergy}
			W(x, u, A) \ := \ \bar{W}(x,u,\phi(F)).
		\end{equation}
	\end{definition}

    \begin{definition}[Quasiconvex energy]\label{Def: qConvexEnergy}
		An energy $W: \Om \times \R^n \times \R^{n \times n} \rightarrow \R$ is said to be \textbf{quasiconvex} if for a.e. $x \in \Om$, all $u \in \R^n$, and $A \in \R^{n \times n}$ the following inequality holds for all $\vphi \in W^{1, \infty}_0((0, 1)^n; \R^n)$:
		\begin{equation}\label{eq: qConvexEnergy}
			W(x, u, A) \ \leq \ \int_{(0, 1)^n}W(x, u, A + \grad \vphi(y))dy.
		\end{equation}
	\end{definition}

   Now, we provide the assumptions needed for our energy density $W$.

\begin{assumption}[Assumptions on energy density]
 \label{Assump: energyDensityAssump}
 The energy density $W: \Om \times \R^n \times \R^{n \times n} \to \R \cup \{ \infty \}$ is quasiconvex (in the sense of Definition \ref{Def: qConvexEnergy}) and is assumed to satisfy:
 \begin{itemize}
     \item $W$ is $\mathcal{L}^n \times \mathcal{B}^n \times \mathcal{B}^{n \times n}$-measurable and is a Carathéodory integrand 
\item $W (x, \cdot, \cdot)$ is lower semi-continuous for a.e.\ $x \in \Om$.
\item $W$ satisfies the growth bounds 
 \begin{equation} \label{Eq: WBound}
     c_w|A|^p - c_0 \ \leq \ W(x, z, A) \ \leq \ C_w(1 + |z|^p + |A|^p)
 \end{equation}
 for some $c_w, c_0, C_w > 0$, $p\in (1,\infty)$, a.e. $x \in \Om$, and all $A \in \R^{n \times n}$.
 \end{itemize}
 \end{assumption}

 We stress that $W$ does not need to be [Fréchet] differentiable for our analysis, as at no point do we need to introduce an Euler-Lagrange system associated with $W$. We may also now define our nonlocal energy density $\mathcal{W}^{s, \de}_g: L^p(\Om; \R^n) \rightarrow \R$, for fixed $\de > 0$, $s\in (0,1)$, and $g \in L^{p'}(\Om; \R^n)$ as follows:
	\begin{equation}\label{Eq: NLEnergy}
		\mathcal{W}^{\de, s}_g(u) \ := \ \begin{cases}
			\int_{\Om}W(x, u(x), D^s_{\de}u(x))dx - \lang g, u\rang, \ u \in H^{s, p, \de}(\Om; \R^n) \\
			+\infty \ \text{otherwise},
		\end{cases}
	\end{equation}

The key energy minimization result for quasiconvex energy densities is \cite[Corollary 2]{cueto2023variational}, which can be applied to our energy $\mathcal{W}^{\de, s}_g(\cdot)$ since $W$ is quasiconvex and adheres to \eqref{Eq: WBound}, $s \in (0, 1)$ and $\de > 0$ are fixed, and $\Lin^n(\pa \Om_{-\de}) = 0$.

\begin{proposition}[Existence of minimizers for nonlocal energy density]\label{Prop: QuasiConvexityEnergy}
    Fix $g \in Z_{\text{ad}}$. Then there exists a $u_{\de, s} \in H^{s, p, \de}_0(\Om_{-\de}; \R^n)$ such that $u_{\de, s} \in \argmin_{v \in H^{s, p, \de}_0(\Om_{-\de}; \R^n)}\mathcal{W}^{\de, s}_g(v)$.
\end{proposition}

Such a result can also readily be established for the local problem that is introduced in Subsection \ref{Subsec: LProblemFormulation}.

Since we focus on the case where $W$ is quasiconvex, we do not write out the analogous version for a polyconvex energy density here (with weaker growth assumptions), but merely mention it can be found as \cite[Theorem 6.1]{bellido2023minimizers}. Now we provide the rigorous definition of an $\ep$-quasi-minimum. 

\begin{definition}[$\ep$-quasi-minimum]\label{Def: epQuasiMin}
    Fix $\ep > 0$ and let $E: L^p(\Om; \R^n) \rightarrow \R$ be a proper energy functional, i.e., $\inf_{v \in L^p(\Om; \R^n)}E(v) > -\infty$. We say that $u \in L^p(\Om; \R^n)$ is an \textbf{$\ep$-quasi-minimum} if
    \begin{equation}\label{Eq: epQuasiMin}
        E(u) - \inf_{v \in L^p(\Om; \R^n)}E(v) \ \leq \ \ep.
    \end{equation}
\end{definition}

Since a global minimum is automatically an $\ep$-quasi-minimizer for all $\ep > 0$, we are guaranteed existence of an $\ep$-quasi-minimum whenever Proposition \ref{Prop: QuasiConvexityEnergy} applies.


\subsection{Cost functionals}\label{Subsec: cost}

To provide assumptions on the cost functionals, we first discuss the properties assumed on the admissible class of controls $Z_{\text{ad}}$.

\begin{assumption}[Assumptions on controls]\label{Assump: Zad}
    The admissible set of controls, $Z_{\text{ad}}$, is a closed, nonempty, convex subset of $L^{p'}(\Om; \R^n)$. 
\end{assumption}

\begin{example}[Box constraint]\label{Rmk: box}
    One of the standard choices for $Z_{\text{ad}}$ is the \textbf{box constraint}, i.e., a set of the form
    \begin{equation}\label{Eq: box}
        Z_{\text{ad}} \ := \ \{z \in L^{p'}(\Om; \R^n)  \ | \ \mathfrak{a}(x) \preceq z(x) \preceq \mathfrak{b}(x) \ \text{a.e.} \ x \in \Om\},
	\end{equation}
    where $\mathfrak{a}$ and $\mathfrak{b}$ are fixed vector fields in $L^{\infty}(\Om; \R^n)$.
\end{example}

Now, the main difference as compared to  \cite{cueto2026localization} is that we allow the nonlocal cost functional $\mathcal{F}_{\de, s}(\cdot, \cdot)$ to depend on the nonlocal gradient \eqref{Eq: NLGrad} (and the local cost functional $\mathcal{F}_{\loc}(\cdot, \cdot)$ to depend on the classical gradient).

 \begin{assumption}[Assumptions on cost functionals]\label{Assump: costFuncAssumptions}
        Our nonlocal cost functional $\mathcal{F}_{\de, s}: L^p(\Om; \R^n) \times Z_{\text{ad}} \rightarrow \R$ is of the form
	\begin{equation}\label{Eq: costFuncNL}
		\mathcal{F}_{\de, s}(u, g) \ := \ \int_{\Om}F(x, u(x), D^s_{\de}u(x))dx + \int_{\Om} \La(x) |g(x)|^{p'} dx.
	\end{equation}
    Meanwhile, the local cost functional $\mathcal{F}_{\loc}: L^p(\Om; \R^n) \times Z_{\text{ad}}$ is of the form
    \begin{equation}\label{Eq: costFuncL}
		\mathcal{F}_{\loc}(u, g) \ := \ \int_{\Om}F(x, u(x), \grad u(x))dx + \int_{\Om} \La(x) |g(x)|^{p'} dx.
	\end{equation}
    It is possible $\mathcal{F}_{\de, s}(\cdot, \cdot)$ takes the value $+\infty$ if $u \in L^p(\Om; \R^n)\setminus H^{s,p,\de}(\Om; \R^n)$, and similarly for $\mathcal{F}_{\loc}(\cdot, \cdot)$. Additionally, we use the following assumptions on $F$, $\Lambda$, and $g$:
    \begin{enumerate}
     \item $\La \in L^{\infty}(\Om)$ is a strictly positive function, i.e. there exists $\la > 0$ such that $\La(x) \geq \la$ for all $x \in \Om$. In this setting we note that
     \begin{equation}\label{Eq: weightNormLp'}
         \|g\|_{L^{p'}_{\La}(\Om; \R^n)} \ := \ \left(\int_{\Om}|g(x)|^{p'}dx\right)^{\frac{1}{p'}}
     \end{equation}
     defines a norm on $L^{p'}(\Om; \R^n)$.
    \item $F: \Om \times \R^n \times \R^{n \times n} \rightarrow \R$ is a Carathéodory function satisfying the following assumptions:
	\begin{enumerate}
		\item For all $z \in \R^n$ and $A \in \R^{n \times n}$, the mapping $x \mapsto F(x, z, A)$ is measurable;
		\item For all $x \in \Om$ and $z \in \R^n$ the mapping $A \mapsto F(x, z, A)$ is  quasiconvex;
		\item There exist constants $c_F, a_0, C_F > 0$ such that
		\begin{equation}\label{Eq: CostBounds}
			c_F|A|^p - a_0 \ \leq \ F(x, z, A) \ \leq \ C_F(1 + |z|^p + |A|^p).
		\end{equation}
    \end{enumerate}
    \end{enumerate}
    
    \end{assumption}

	\begin{example}[Desired state]\label{Ex: validCost}
    Consider a cost functional of the form
	\begin{equation}\label{Eq: costFuncSpec}
		\mathcal{F}_{\de, s}(u, g) \ := \ \frac{1}{p}\|u - u_{\des}\|^p_{L^p(\Om; \R^n)} + \frac{\la}{p'}\|g\|^{p'}_{L^{p'}(\Om; \R^n)}
	\end{equation}
    for all $\de > 0$ and $s \in (0, 1)$, where $\mathcal{F}_{\loc}(\cdot, \cdot)$ is defined in the same way. Here $\la > 0$ is a regularization parameter, and $u_{\des} \in L^p(\Om; \R^n)$ is fixed. This cost functional clearly satisfies Assumption \ref{Assump: costFuncAssumptions}.
	\end{example}

    \begin{remark}[Compliance]\label{Ex: Compliance}
    Consider the cost functional of compliance type 
    \begin{equation}\label{Eq: compliance}
    \mathcal{F}_{\de, s}(u, g) \ := \ \mathcal{F}_{\loc}(u, g) \ := \ \int_{\Om}g(x) \cdot u(x)dx + \frac{\la}{p'}\|g\|^{p'}_{L^{p'}_{\La}(\Om; \R^n)},
    \end{equation}
    which was the focus in the nonlocal optimal design papers \cite{andres2023minimization, mengesha2025design}. While \eqref{Eq: compliance} does not satisfy Assumption \ref{Assump: costFuncAssumptions}, the main convergence and existence theorems will still hold because the compliance term exhibits weak-strong convergence under the natural topologies of convergence for the states and controls.

\end{remark}

\subsection{Nonlocal problem formulations}\label{Subsec: setup}

We have two different general problem statements: the nonlocal and local versions of the control problem with an $\ep$-quasi minimization constraint (see Definition \ref{Def: epQuasiMin}). Here we denote $m_{g, \de, s} := \min_{H^{s, p, \de}_0(\Om_{-\de}; \R^n)}\mathcal{W}^{\de, s}_g(\cdot)$.

\begin{problem}[Nonlocal optimal control problem with quasi-minimization constraint]\label{Pr: NLQuasiMinimizers}
    Our nonlocal optimal control problem with an $\ep$-quasi minimization constraint is to solve
    \begin{equation}\label{Eq: stateMinNLQuasi}
			\mathcal{F}_{\de, s}(\overline{u_{\de, s}}, \overline{g_{\de, s}}) \ = \ \min_{(u_{\de, s}, g_{\de, s}) \in \mathbb{T}^{\de, s}_{\ep}}\mathcal{F}_{\de, s}(u_{\de, s}, \ g_{\de, s}),
		\end{equation} 
        where the class of admissible pairs is given by
        \begin{equation}\label{Eq: nonlocalAdmissSetQuasiEp}
    \mathbb{T}^{\de, s}_{\ep} \ := \ \{(u, g) \in H^{s, p, \de}_0(\Om_{-\de}; \R^n) \times Z_{\text{ad}}, \ \mathcal{W}^{\de, s}_g(u) - m_{g, \de, s} \ \leq \ \ep\}.
\end{equation}
\end{problem}


\section{Existence theory}\label{Sec: ExistenceQuasi}

The goal of this section is to prove existence of solutions to Problem \ref{Pr: NLQuasiMinimizers}. Since the energy to be minimized as a constraint may not have a unique minimizer (due to non-convexity), we cannot, as is usually the case for control problems, define a solution mapping. This first lemma, a variant of Lemma 3.6 in \cite{cueto2026localization}, serves as an analogue to the stability estimate that is typically proven for PDEs with unique solutions.

    \begin{lemma}[Boundedness of admissible states]\label{Lem: boundednessAdmissibleStatesNLQuasi}
        Let $W: \Om \times \R^n \times \R^{n \times n} \rightarrow \R$ be an energy density satisfying Assumption \ref{Assump: energyDensityAssump}, and let $\ep > 0$ be fixed. Then, the image of $\mathbb{T}^{\de, s}_{\ep}$ into $H^{s, p, \de}_0(\Om_{-\de}; \R^n)$, i.e., the collection of states which are $\ep$ quasi-minimizers of $\mathcal{W}^{\de, s}_g(\cdot)$ for some $g \in Z_{\text{ad}}$ given by
	\begin{equation}\label{Eq: UNLAdmisQuasi}
		V^{\de, s}_{\ep} \ := \ \{u_{\de, s} \in H^{s, p, \de}_0(\Om_{-\de}; \R^n) \ | \ \ex g_{\de, s} \in Z_{\text{ad}}, \ (u_{\de, s}, g_{\de, s}) \in \mathbb{T}^{\de, s}_{\ep}\},
	\end{equation}
	is a bounded subset of $H^{s, p, \de}_0(\Om_{-\de}; \R^n)$. 
    \end{lemma}

    \begin{proof}
    Let $u_{\de, s} \in V^{\de, s}_{\ep}$ be arbitrary. Then there exists $g_{\de, s} \in Z_{\text{ad}}$ so that
    \begin{equation}\label{Eq: UNLAdmisQuasiEq1}
        \mathcal{W}^{\de, s}_{g_{\de, s}}(u_{\de, s})  \ \leq \ \min \mathcal{W}^{\de, s}_{g_{\de, s}}(\cdot) + \ep.
    \end{equation}
    Pick $\widehat{u_{\de, s}} \in H^{s, p, \de}_0(\Om_{-\de}; \R^n)$ so that $\min \mathcal{W}^{\de, s}_{g_{\de, s}}(\cdot) = \mathcal{W}^{\de, s}_{g_{\de, s}}(\widehat{u_{\de, s}})$; note this function exists due to Proposition \ref{Prop: QuasiConvexityEnergy}. Moreover, due to the structure of the energy density and \eqref{Eq: WBound}, $\mathcal{W}^{\de, s}_{g_{\de, s}}(\widehat{u_{\de, s}}) \leq C_w|\Om|$, so 
    \begin{equation}\label{Eq: UNLAdmisQuasiEq2}
         \mathcal{W}^{\de, s}_{g_{\de, s}}(u_{\de, s})  \ \leq \ C_w|\Om| + \ep.
    \end{equation}
    Using \eqref{Eq: WBound} again we get
    \begin{equation}\label{Eq: UNLAdmisQuasiEq3}
        c_w\|D^s_{\de}u_{\de, s}\|^p_{L^p(\Om; \R^{n \times n})} \ \leq \ (c_0 + C_w)|\Om| + \ep,
    \end{equation}
    completing the proof.
    \end{proof}

    To carry out the existence proof with the direct method, we also need a weak closedness result for $\mathbb{T}^{\de, s}_{\ep}$ (as a subset of $H^{s, p, \de}(\Om; \R^n) \times L^{p'}(\Om; \R^n)$). 

 \begin{lemma}[Weak closedness]\label{Lem: NLClosureOfMinimizationQuasi}
         Let $W: \Om \times \R^n \times \R^{n \times n} \rightarrow \R$ be an energy density satisfying Assumption \eqref{Assump: energyDensityAssump}. Then the set $\mathbb{T}^{\de, s}_{\ep}$ is weakly closed in $H^{s, p, \de}(\Om; \R^n) \times L^{p'}(\Om; \R^n)$.
    \end{lemma}

    \begin{proof}
    Let $\{(u_j, g_j)\}^{\infty}_{j = 1} \subset \mathbb{T}^{\de, s}_{\ep}$ be a sequence such that $u_j \rightharpoonup u$ weakly in $H^{s, p, \de}(\Om; \R^n)$ and $g_j \rightharpoonup g$ weakly in $L^{p'}(\Om; \R^n)$. We must show that $(u, g) \in \mathbb{T}^{\de, s}_{\ep}$, i.e.,
    \begin{equation}\label{Eq: NLClosureOfMinimizationQuasiEq1}
        \mathcal{W}^{\de, s}_g(u) - \min \mathcal{W}^{\de, s}_g(\cdot) \ \leq \ \ep.
    \end{equation}
    To this end, notice that for all $j \in \N^+$ we have
     \begin{equation}\label{Eq: NLClosureOfMinimizationQuasiEq2}
        \mathcal{W}^{\de, s}_{g_j}(u_j) - \min \mathcal{W}^{\de, s}_{g_j}(\cdot) \ \leq \ \ep.
    \end{equation}
    Due to the convergence properties of our sequences and weak lower semi-continuity of $W$, we have
    \begin{equation}\label{Eq: NLClosureOfMinimizationQuasiEq3}
        \liminf_{j \rightarrow \infty}\mathcal{W}^{\de, s}_{g_j}(u_j) \ \geq \  \mathcal{W}^{\de, s}_g(u).  
    \end{equation}
   Now we show that
 \begin{equation}\label{Eq: NLClosureOfMinimizationQuasiEq4}
        \lim_{j \rightarrow \infty}\min \mathcal{W}^{\de, s}_{g_j}(\cdot) \ = \ \min \mathcal{W}^{\de, s}_g(\cdot).
    \end{equation}
    Due to Proposition \ref{Prop: QuasiConvexityEnergy} applied for each $j \in \N^+$, there exists a sequence $\{v_j\}^{\infty}_{j = 1} \subset H^{s, p, \de}_0(\Om_{-\de}; \R^n)$ so that $\mathcal{W}^{\de,  s}_{g_j}(v_j) = \min\mathcal{W}^{\de, s}_{g_j}(\cdot)$. Then by Lemma \ref{Lem: boundednessAdmissibleStatesNLQuasi}, the family $\{v_j\}^{\infty}_{j = 1}$ is bounded in $H^{s, p, \de}_0(\Om_{-\de}; \R^n)$.  By reflexivity, there exists a (not relabeled) sub-sequence $\{v_j\}^{\infty}_{j = 1}$ and a limit $v \in H^{s, p, \de}_0(\Om_{-\de}; \R^n)$ so that $v_j \rightharpoonup v$ weakly in $H^{s, p, \de}(\Om; \R^n)$. Moreover, for each $j \in \N^+$ and $w \in H^{s, p, \de}_0(\Om_{-\de}; \R^n)$ arbitrary we have
    \begin{equation}\label{Eq: NLClosureOfMinimizationQuasiEq5}
        \mathcal{W}^{\de, s}_{g_j}(v_j) \ \leq \ \mathcal{W}^{\de, s}_{g_j}(w).
    \end{equation}
    Then by weak lower semi-continuity of $W$ and the weak convergence $g_j \rightharpoonup g$ in $L^{p'}(\Om; \R^n)$, we have
    \begin{equation}\label{Eq: NLClosureOfMinimizationQuasiEq6}
        \mathcal{W}^{\de,  s}_g(v) \ \leq \ \mathcal{W}^{\de, s}_g(w).
    \end{equation}    
    This means $\min \mathcal{W}^{\de, s}_g(\cdot) = \mathcal{W}^{\de, s}_g(v)$, and we conclude
    \begin{equation}\label{Eq: NLClosureOfMinimizationQuasiEq7}
       \min \mathcal{W}^{\de, s}_g(\cdot) \ \leq \ \liminf_{j \rightarrow \infty}\min \mathcal{W}^{\de, s}_{g_j}(\cdot).
    \end{equation}
To make \eqref{Eq: NLClosureOfMinimizationQuasiEq7} an equality, we may apply \eqref{Eq: NLClosureOfMinimizationQuasiEq5} by setting $w := v$ and sending $j \rightarrow \infty$, hence proving \eqref{Eq: NLClosureOfMinimizationQuasiEq4}. Finally, combining \eqref{Eq: NLClosureOfMinimizationQuasiEq2}, \eqref{Eq: NLClosureOfMinimizationQuasiEq3}, and \eqref{Eq: NLClosureOfMinimizationQuasiEq4} gives \eqref{Eq: NLClosureOfMinimizationQuasiEq1}, as desired.
    \end{proof}

    Third, we need a lower semi-continuity result for the density function $F$ (introduced in Assumption \ref{Assump: costFuncAssumptions}).

    \begin{lemma}\label{Lem: lscDirect}
If $\de > 0$ and $s \in (0, 1)$ are fixed and $\{u_j\}^{\infty}_{j = 1} \subset H^{s, p, \de}_0(\Om_{-\de}; \R^n)$ is a sequence such that $u_j \rightarrow u$ strongly in $L^p(\Om; \R^n)$ and $D^s_{\de} u_j \rightharpoonup D^s_{\de} u$ weakly in $L^p(\Om; \R^{n \times n})$, then
            \begin{equation}\label{Eq: lscF}
                \liminf_{j \rightarrow \infty}\int_{\Om}F(x, u_j(x), D^s_{\de}u_j(x))dx \ \geq \ \int_{\Om}F(x, u(x), D^s_{\de}u(x))dx.
            \end{equation}
    \end{lemma}

    \begin{proof}
        This follows from the conditions on $F$ given in Assumption \ref{Assump: costFuncAssumptions} and the characterization result \cite[Theorem 5]{cueto2023variational}. 
    \end{proof}

Finally, we may state and prove the main existence result.

\begin{theorem}[Existence for non-local optimal control problem]\label{Thm: ExistenceControlNonlocalQuasi}
        Suppose $\de > 0$, $s \in (0, 1)$, and $\ep > 0$ are fixed, and that $\mathcal{F}_{\de, s}: L^p(\Om; \R^n) \times Z_{\text{ad}} \rightarrow \R$ satisfies Assumption \ref{Assump: costFuncAssumptions}. Then there exists a pair $(\overline{u_{\de, s}}, \overline{g_{\de, s}}) \in \mathbb{T}^{\de, s}_{\ep}$ that solves Problem \ref{Pr: NLQuasiMinimizers}.
    \end{theorem}

    \begin{proof}
        The first step is to prove that $\mathcal{F}_{\de, s}(\cdot, \cdot)$ is bounded from below uniformly on $\mathbb{T}^{\de, s}_{\ep}$. Let $(u, g) \in \mathbb{T}^{\de, s}_{\ep}$ be arbitrary and then due to \eqref{Eq: CostBounds},
        \begin{equation}\label{Eq: ExistenceControlNonlocalQuasiEq1}
            \mathcal{F}_{\de, s}(u, g) \ \geq \ c_F\|D^s_{\de}u\|^p_{L^p(\Om; \R^{n \times n})} - a_0|\Om|.
        \end{equation}
        The right-hand side of \eqref{Eq: ExistenceControlNonlocalQuasiEq1} is uniformly bounded from below due to Lemma \ref{Lem: boundednessAdmissibleStatesNLQuasi}. Denote $m := \inf_{\mathbb{T}^{\de, s}_{\ep}}\mathcal{F}_{\de, s}(\cdot, \cdot)$ and let $\{(u_j, g_j)\}^{\infty}_{j = 1} \subset \mathbb{T}^{\de, s}_{\ep}$ be a sequence such that $\lim_{j \rightarrow \infty}\mathcal{F}_{\de, s}(u_j, g_j) \ = \ m$. Again using Lemma \ref{Lem: boundednessAdmissibleStatesNLQuasi}, the sequence $\{u_j\}^{\infty}_{j = 1}$ is a bounded subset of $H^{s, p, \de}_0(\Om_{-\de}; \R^n)$. Thus by reflexivity there exists a $\overline{u_{\de}} \in H^{s, p, \de}_0(\Om_{-\de}; \R^n)$ so that $u_j \rightharpoonup \overline{u_{\de, s}}$ weakly in $H^{s, p, \de}(\Om; \R^n)$. Furthermore, $\{g_j\}^{\infty}_{j = 1}$ is a subset of $Z_{\text{ad}}$, which is itself a bounded subset of $L^{p'}(\Om; \R^n)$, so there exists a $\overline{g_{\de, s}} \in Z_{\text{ad}}$ and a further (not relabeled) sub-sequence so that $g_j \rightharpoonup \overline{g_{\de, s}}$ weakly in $L^{p'}(\Om; \R^n)$. 

        Due to Lemma \ref{Lem: NLClosureOfMinimizationQuasi}, we have that $(\overline{u_{\de, s}}, \overline{g_{\de, s}}) \in \mathbb{T}^{\de, s}_{\ep}$. Moreover, \eqref{Eq: weightNormLp'} defines a norm on $L^{p'}(\Om; \R^n)$, so the regularization term of the cost functional exhibits weak lower semi-continuity. This observation combined with Lemma \ref{Lem: lscDirect} yields
        \begin{equation}\label{Eq: ExistenceControlNonlocalQuasiEq2}
            m \ = \ \lim_{j \rightarrow \infty}\mathcal{F}_{\de, s}(u_j, g_j) \ \geq \ \mathcal{F}_{\de, s}(\overline{u_{\de, s}}, \overline{g_{\de, s}}) \ \geq \ m,
        \end{equation}
        which means that $(\overline{u_{\de, s}}, \overline{g_{\de, s}})$ is a solution to Problem \ref{Pr: NLQuasiMinimizers}.
    \end{proof}


\section{Asymptotic Analysis}\label{Sec: quasiMinimizersAsymptotics}

In this section we detail our convergence results for Problem \ref{Pr: NLQuasiMinimizers}, both when $s \in (0, 1)$ is fixed and $\de \rightarrow 0^+$; and when $\de > 0$ is fixed and $s \rightarrow 1^-$. In both cases we expect to approximate Problem \ref{Pr: LQuasiMinimizers}. However, for Problem \ref{Pr: NLQuasiMinimizers} we assume that admissible states vanish on $\Om_{\de}\setminus \Om_{-\de}$, and this set shrinks when we take $\de \rightarrow 0^+$ but not when we take $s \rightarrow 1^-$. Thus the underlying domain where states for Problem \ref{Pr: LQuasiMinimizers} vanish depends on which limit is taken; we use $\widetilde{\Om}$ to signify this domain in both asymptotic regimes:
\begin{itemize}
    \item In Subsection \ref{Subsec: quasiMinimizersAsymptoticss->1}, take $\widetilde{\Om} := \Om_{-\de}$.
    \item In Subsection \ref{Subsec: QuasiMinsConvDe->0}, take $\widetilde{\Om} := \Om$.
\end{itemize}


\subsection{Local problem formulation}\label{Subsec: LProblemFormulation}

We first need a local version of \eqref{Eq: NLEnergy}. We define $\mathcal{W}^{\loc}_g: L^p(\widetilde{\Om}; \R^n) \rightarrow \R \cup \{+\infty\}$, for $g\in Z_{\text{ad}}$ fixed, as
	\begin{equation}\label{Eq: GenEnergyL}
		\mathcal{W}^{\loc}_g(u)\ := \ \begin{cases}
			\int_{\widetilde{\Om}}W(x, u(x), \grad u(x))dx - \lang g, u\rang, \ u \in W^{1, p}(\widetilde{\Om}; \R^n) \\
			+\infty  \ \text{otherwise},
		\end{cases}
	\end{equation}
    where $W$ is given as in Subsection \ref{Subsec: energy}.

Now, as a counterpart to Problem \ref{Pr: NLQuasiMinimizers}, we introduce a local version of the control problem with an $\ep$-quasi minimization constraint (see Definition \ref{Def: epQuasiMin}). Here we denote $m_{g, \loc} := \min_{W^{1, p}_0(\widetilde{\Om}; \R^n)}\mathcal{W}^{\loc}_g(\cdot)$.

\begin{problem}[Local optimal control problem with quasi-minimization constraint]\label{Pr: LQuasiMinimizers}
    Our local optimal control problem with an $\ep$-quasi minimization constraint is to solve
    \begin{equation}\label{Eq: stateMinLQuasi}
			\mathcal{F}_{\loc}(\overline{u}, \overline{g}) \ = \ \min_{(u, g) \in \mathbb{T}^{\loc}_{\ep}}\mathcal{F}_{\loc}(u_{\de, s}, g_{\de, s}),
		\end{equation} 
        where the class of admissible pairs is given by
        \begin{equation}\label{Eq: localAdmissSetQuasiEp}
    \mathbb{T}^{\loc}_{\ep} \ := \ \{(u, g) \in W^{1, p}_0(\widetilde{\Om}; \R^n) \times Z_{\text{ad}}, \ \mathcal{W}^{\loc}_g(u) - m_{g, \loc} \ \leq \ \ep\}
\end{equation}
\end{problem}

 To prove existence of solutions to Problem \ref{Pr: LQuasiMinimizers}, we again need to establish weak lower semi-continuity of the cost functional. This time we use the following result instead of Lemma \ref{Lem: lscDirect}.

 \begin{lemma}\label{Lem: lscDirectLocal}
     If  $\{u_j\}^{\infty}_{j = 1} \subset W^{1, p}(\Om; \R^n)$ is a sequence such that $u_j \rightarrow u$ strongly in $L^p(\Om; \R^n)$ and $\grad u_j \rightharpoonup \grad u$ weakly in $L^p(\Om; \R^{n \times n})$, then
            \begin{equation}\label{Eq: lscFloc}
                \liminf_{j \rightarrow \infty}\int_{\Om}F(x, u_j(x), \grad u_j(x))dx \ \geq \ \int_{\Om}F(x, u(x), \grad u(x))dx.
            \end{equation}
 \end{lemma}

 Compared to the proof of Lemma \ref{Lem: lscDirect}, instead of using the characterization result \cite[Theorem 5]{cueto2023variational}, we use \cite[Theorem 8.11]{dacorogna2007direct} (which uses the assumption that $F$ is a Carathéodory integrand). 

Now we state the local version of Theorem \ref{Thm: ExistenceControlNonlocalQuasi}; notice this implicitly uses a local version of Theorem \ref{Prop: QuasiConvexityEnergy} for $\mathcal{W}^{\loc}_g(\cdot)$.

        \begin{theorem}[Existence for local problem]\label{Thm: ExistenceControlLocalQuasi}
        Suppose  $\ep > 0$ is fixed, and that $\mathcal{F}_{\loc}: L^p(\Om; \R^n) \times Z_{\text{ad}} \rightarrow \R$ satisfies Assumption \ref{Assump: costFuncAssumptions}. Then there exists a pair $(\overline{u}, \overline{g}) \in \mathbb{T}^{\loc}_{\ep}$ that solves Problem \ref{Pr: LQuasiMinimizers}.
    \end{theorem}

    \begin{remark}
        For Theorems \ref{Thm: ExistenceControlNonlocalQuasi} and \ref{Thm: ExistenceControlLocalQuasi}, the lower bound of \eqref{Eq: CostBounds} may be replaced by the more general bound
        \begin{equation}\label{Eq: CostBoundLStringent}
            F(x, z, A) \ \geq \ -C_F(1 + |z|^p + |A|^q)
        \end{equation}
        for some $q \in [1, p)$. The results cited for proving Lemma \ref{Lem: lscDirectLocal} may still be invoked with this lower bound. However, \eqref{Eq: CostBoundLStringent} will not be sufficient for carrying out the asymptotic analysis of Subsections \ref{Subsec: quasiMinimizersAsymptoticss->1} and \ref{Subsec: QuasiMinsConvDe->0}.
    \end{remark}


\subsection{Convergence as $s \rightarrow 1^-$}\label{Subsec: quasiMinimizersAsymptoticss->1}

Before we can study convergence of solutions for Problem \ref{Pr: NLQuasiMinimizers} as $s \rightarrow 1^-$, we must consider convergence of the energies. This convergence will require one additional assumption on the kernel $w_{\de}$.

\begin{assumption}\label{Assump: kernels->1}
    The constant $a_0$ from Assumption \ref{Assump: Kernels} equals $1$.
\end{assumption}

This assumption is used to obtain a Nonlocal Poincaré Inequality with constant independent of $s$ (\cite[Theorem 4]{cueto2023variational}).

\begin{proposition}[Nonlocal Poincaré inequality I]\label{Prop: NLPoincares}
		Let $s \in [0, 1]$, $p \in (1, \infty)$, $\Om \subset \R^n$ be open and bounded. Then there exists a constant $C > 0$ depending only on $\Om$, $\de$, and $p$ such that for all $u \in H^{s, p, \de}_0(\Om_{-\de}; \R^n)$, we have
		\begin{equation}\label{Eq: NLPoincares}
			\|u\|_{L^p(\Om; \R^n)} \ \leq \ C\|D^s_{\de}u\|_{L^p(\Om; \R^{n\times n})}.
		\end{equation}
	\end{proposition}

Now, the usual notion of convergence for related classes of problems is $\Ga$-convergence; we opt to simply provide the $\Ga$-convergence result for our setting rather than provide an abstract definition.

\begin{theorem}[$\Ga$-convergence as $s \rightarrow 1^-$]\label{Thm: Gas->1Theorem}
	Let $\de > 0$, $g \in Z_{\text{ad}}$ be fixed. Then the family of functionals $\{\mathcal{W}^{\de, s}_g\}_{s < 1}$ will $\Ga$-converge in the strong $L^p(\Om; \R^n)$ topology to $\mathcal{W}^{\loc}_g$, which we will denote $\mathcal{W}^{\de, s}_g \xrightarrow{\Ga, \ s \rightarrow 1^-} \mathcal{W}^{\loc}_g$. In other words, we have the following:
	\begin{enumerate}[label=\textbf{GCs\arabic*}]
		\item\label{GCs1} If $\{u_s\}_{s < 1} \subset L^p(\Om; \R^n)$ is a sequence such that $u_s \rightarrow u$ strongly in $L^p(\Om; \R^n)$, then we have the \textbf{lim-inf inequality}
		\begin{equation}\label{Gas->1Liminf}
			\mathcal{W}^{\loc}_g(u) \ \leq \ \liminf_{s \rightarrow 1^-}\mathcal{W}^{\de, s}_g(u_s).
		\end{equation}
		\item\label{GCs2} If $u \in L^p(\Om; \R^n)$, then there exists a \textbf{recovery sequence} of $\{u_s\}_{s < 1} \subset L^p(\Om; \R^n)$ such that $u_s \rightarrow u$ strongly in $L^p(\Om; \R^n)$ and
		\begin{equation}\label{Gas->1Limsup}
			\mathcal{W}^{\loc}_g(u) \ \geq \ \limsup_{s \to 1^-} \mathcal{W}^{\de, s}_g(u_s).
		\end{equation}
	\end{enumerate}
\end{theorem}
We comment this is precisely Theorem 4.8 in \cite{cueto2026localization}, but the core mechanism of its proof was developed in \cite[Theorem 7]{cueto2023variational}. The main features were a translation mapping between nonlocal and local gradients, and the equivalence between quasiconvexity and weak lower semi-continuity on $H^{s, p, \de}(\Om; \R^n)$ (Theorems 2 and 5 in \cite{cueto2023variational}, respectively). We also comment that \eqref{Gas->1Limsup} was proven with a constant recovery sequence, i.e., $u_s := u$, and \eqref{Gas->1Limsup} was actually an equality:
\begin{equation}\label{Gas->1Limeq}
    \mathcal{W}^{\loc}_g(u) \ = \ \lim_{s \to 1^-} \mathcal{W}^{\de, s}_g(u).
\end{equation}
Furthermore, the energies $\{\mathcal{W}^{\de, s}_g(\cdot)\}_{s < 1}$ are equi-coercive with respect to the strong topology on $L^p(\Om; \R^n)$ due to \eqref{Eq: WBound}. Finally, similar results were obtained in \cite{bellido2020bond, bellido2021gamma, mengesha2015variational} for various strongly coupled systems of nonlocal equations.

To prove our main convergence result, we also need a version of Lemma \ref{Lem: boundednessAdmissibleStatesNLQuasi} for varying $s$. 

 \begin{lemma}[Boundedness of admissible states as $s \rightarrow 1^-$]\label{Lem: boundednessAdmissibleStatesNLQuasis->1}
        Let $W: \Om \times \R^n \times \R^{n \times n} \rightarrow \R$ be an energy density satisfying Assumption \ref{Assump: energyDensityAssump}, and let $\ep > 0$, $\de > 0$ be fixed. Let $\{(u_s, g_s)\}_{s < 1}$ be a family of pairs such that $(u_s, g_s) \in \mathbb{T}^{\de, s}_{\ep}$ for all $s < 1$. Then $\liminf_{s \rightarrow 1^-}\|D^s_{\de}u_s\|_{L^p(\Om; \R^{n \times n})} < \infty$.
    \end{lemma}

Finally, we need a version of Lemma \ref{Lem: lscDirect} when $s$ varies.

\begin{lemma}\label{Lem: lscs->1}
    If $\de > 0$ is fixed and $\{u_s\}_{s < 1}$ is a family such that $u_s \in H^{s, p, \de}(\Om; \R^n)$ for all $s < 1$, $\liminf_{s \rightarrow 1^-}\|D^s_{\de} u_s\|_{L^p(\Om; \R^{n \times n})} < \infty$, $u_s \rightarrow u$ strongly in $L^p(\Om; \R^n)$, and $D^s_{\de} u_s \rightharpoonup \grad u$ weakly in $L^p(\Om; \R^{n \times n})$, then
            \begin{equation}\label{Eq: lscFs->1}
                \liminf_{s \rightarrow 1^-}\int_{\Om}F(x, u_s(x), D^s_{\de}u_s(x))dx \ \geq \ \int_{\Om}F(x, u(x), \grad u(x))dx.
            \end{equation}
\end{lemma}

\begin{proof}
    It suffices to prove each of the following:
    \begin{equation}\label{Eq: mainIntegrands->1}
        \liminf_{s \rightarrow 1^-}\int_{\Om_{-\de}}F(x, u_s(x), D^s_{\de}u_s(x))dx \ \geq \ \int_{\Om_{-\de}}F(x, u(x), \grad u(x))dx.
    \end{equation}
    \begin{equation}\label{Eq: leftovers}
        \liminf_{s \rightarrow 1^-}\int_{\Om\setminus \Om_{-\de}}F(x, u_s(x), D^s_{\de}u_s(x))dx \ \geq \ \int_{\Om\setminus \Om_{-\de}}F(x, u(x), \grad u(x))dx.
    \end{equation}

First, to prove \eqref{Eq: mainIntegrands->1}, notice that our growth condition \eqref{Eq: CostBounds} is a special case of the growth condition used for \cite[Theorem 8.11]{dacorogna2007direct}. Thus we may repeat the argument of (6.5) in \cite{cueto2023variational}, which involves a translation operation between nonlocal and classical gradients, to obtain \eqref{Eq: mainIntegrands->1}. While that paper does not consider integrands with $u$-dependence, we may adapt the same argument since \cite[Theorem 8.11]{dacorogna2007direct}  applies for quasiconvex energy densities with $u$-dependence (also see \cite[Theorem 7.5]{fonseca2007modern} or \cite[Theorem 5.4]{ball1981null} for the case where $F$ is convex, and  \cite[Theorem 5.20]{Rin} as an alternative source for the quasiconvex case).
    
    Meanwhile, to prove \eqref{Eq: leftovers}, we may repeat the Fatou's lemma argument used on (6.6) in \cite{cueto2023variational} since $F(\cdot, \cdot, \cdot)$ is bounded from below by a constant (see \eqref{Eq: CostBounds})

\end{proof}

Now we state and prove our main convergence of minimizers result. The proof idea is as follows: identify a candidate minimizing control-state pair for Problem \ref{Pr: LQuasiMinimizers} using compactness arguments, then prove this pair is admissible using Theorem \ref{Thm: Gas->1Theorem}. This second step is more technical and amounts to delicately constructing a recovery sequence. If $(u, g) \in \mathbb{T}^{\loc}_{\ep}$, then the values of the differences $\mathcal{W}^{\de, s}_g(u_s) - \mathcal{W}^{\de, s}_g(v_s)$, where $\{u_s\}_{s < 1}$ is a candidate recovery sequence for $u$ and $v_s \in \argmin_{v \in H^{s, p, \de}_0(\Om_{-\de}; \R^n)}\mathcal{W}^{\de, s}_g(v)$, may oscillate and actually exceed $\ep$ for some values of $s$. With this in mind we break into cases depending on whether $\mathcal{W}^{\loc}_g(u) - \min_{v \in W^{1, p}_0(\widetilde{\Om}; \R^n)}\mathcal{W}^{\loc}_g(v)$ is strictly less than $\ep$ or equal to $\ep$. In the former case we use a constant recovery sequence for $u$, but in trickier latter case, our recovery sequence is comprised of convex combinations of $u$ and $v_s$.

\begin{theorem}[Convergence of minimizers as $s \rightarrow 1^-$]\label{Thm: QuasiMinimizersConvergenceOfMinss->1}
    Suppose that $(\overline{u_{\de, s}}, \overline{g_{\de, s}}) \in \mathbb{T}^{\de, s}_{\ep}$ is a solution to Problem \ref{Pr: NLQuasiMinimizers} with energy tolerance $\ep$. Then there exists a non-relabeled sub-sequence, a $\overline{u} \in W^{1, p}_0(\widetilde{\Om}; \R^n)$, and a $\overline{g} \in Z_{\text{ad}}$ such that $(\overline{u}, \overline{g}) \in \mathbb{T}^{\loc}_{\ep}$, we have
    \begin{equation}\label{Eq: QuasiMinimizersConvergenceOfMinss->1}
        \lim_{s \rightarrow 1^-}\mathcal{F}_{\de, s}(\overline{u_{\de, s}}, \overline{g_{\de, s}}) \ = \ \mathcal{F}_{\loc}(\overline{u}, \overline{g}),
    \end{equation}
    and that $(\overline{u}, \overline{g})$ is a solution to Problem \ref{Pr: LQuasiMinimizers} with energy tolerance $\ep$. Moreover, $\overline{u_{\de, s}} \rightarrow \overline{u}$ strongly in $L^p(\Om; \R^n)$ and $\overline{g_{\de, s}} \rightharpoonup \overline{g}$ weakly in $L^{p'}(\Om; \R^n)$.
 \end{theorem}

 \begin{proof}
  We first prove that
     \begin{equation}\label{Eq: QuasiMinimizersConvergenceOfMinss->1LimInfEq0}
          \liminf_{s \rightarrow 1^-}\mathcal{F}_{\de, s}(\overline{u_{\de, s}}, \overline{g_{\de, s}}) \ \geq \ \mathcal{F}_{\loc}(\overline{u}, \overline{g}).
     \end{equation}
     To this end, observe that by Lemma \ref{Lem: boundednessAdmissibleStatesNLQuasis->1}, the family $\{\overline{u_{\de, s}}\}_{s < 1}$ is bounded in the sense that 
     \begin{equation}\label{Eq: liminfs->1Conv}
     \liminf_{s \rightarrow 1^-}\|D^s_{\de}\overline{u_{\de, s}}\|_{L^p(\Om; \R^{n \times n})} < \infty.
     \end{equation}
     Then by the compactness result \cite[Lemma 9]{cueto2023variational} there exists $\overline{u} \in W^{1, p}_0(\widetilde{\Om}; \R^n)$ so that, up to a sub-sequence, $\overline{u_{\de, s}} \rightarrow \overline{u}$ strongly in $L^p(\Om; \R^n)$ and $D^s_{\de}\overline{u_{\de, s}} \rightharpoonup \grad \overline{u}$ weakly in $L^p(\Om; \R^{n \times n})$. Thus we may apply Lemma \ref{Lem: lscs->1} to obtain
    \begin{equation}\label{Eq: QuasiMinimizersConvergenceOfMinss->1Controls1}
        \liminf_{s \rightarrow 1^-}\int_{\Om}F(x, \overline{u_{\de, s}}(x), D^s_{\de}\overline{u_{\de, s}}(x))dx \ \geq \ \int_{\Om}F(x, \overline{u}(x), \grad \overline{u}(x))dx.
    \end{equation}
There will also, by reflexivity and weak closedness of $Z_{\text{ad}}$, exist a further sub-sequence and a $\overline{g} \in Z_{\text{ad}}$ such that $\overline{g_{\de, s}} \rightharpoonup \overline{g}$ weakly in $L^{p'}(\Om; \R^n)$. Similarly to the proof of Theorem \ref{Thm: ExistenceControlNonlocalQuasi}, we obtain 
\begin{equation}\label{Eq: QuasiMinimizersConvergenceOfMinss->1Controls1A}
    \liminf_{s \rightarrow 1^-}\|\overline{g_{\de, s}}\|^{p'}_{L^{p'}_{\La}(\Om; \R^n)} \ \geq \ \|\overline{g}\|^{p'}_{L^{p'}_{\La}(\Om; \R^n)}
\end{equation}
Combining \eqref{Eq: QuasiMinimizersConvergenceOfMinss->1Controls1} and \eqref{Eq: QuasiMinimizersConvergenceOfMinss->1Controls1A} gives us \eqref{Eq: QuasiMinimizersConvergenceOfMinss->1LimInfEq0}.

Now we show that $(\overline{u}, \overline{g}) \in \mathbb{T}^{\loc}_{\ep}$.
Owing to the lim-inf inequality of $\Ga$-convergence (Theorem \ref{Thm: Gas->1Theorem}) and the weak-strong convergence of the compliance term, we have that
\begin{equation}\label{Eq: QuasiMinimizersConvergenceOfMinss->1LimInfEq1}
    \liminf_{s \rightarrow 1^-}\mathcal{W}^{\de, s}_{\overline{g_{\de, s}}}(\overline{u_{\de, s}}) \ \geq \ \mathcal{W}^{\loc}_{\overline{g}}(\overline{u}).
\end{equation}
Meanwhile, since the energies $\{\mathcal{W}^{\de, s}_{\overline{g_{\de, s}}}(\cdot)\}_{s < 1}$ are equi-coercive with respect to the strong topology on $L^p(\Om; \R^n)$, we also have that
\begin{equation}\label{Eq: QuasiMinimizersConvergenceOfMinss->1LimInfEq2}
    \lim_{s \rightarrow 1^-}\min_{v \in H^{s, p, \de}_0(\Om_{-\de}; \R^n)}\mathcal{W}^{\de, s}_{\overline{g_{\de, s}}}(v) \ = \ \min_{v \in W^{1, p}_0(\widetilde{\Om}; \R^n)}\mathcal{W}^{\loc}_{\overline{g}}(v),
\end{equation}
by \cite[Corollary 7.20]{dal2012introduction}. Thus, we have that
\begin{equation}\label{Eq: QuasiMinimizersConvergenceOfMinss->1LimInfEq3}
    \mathcal{W}^{\loc}_{\overline{g}}(\overline{u}) - \min_{v \in W^{1, p}_0(\Om; \R^n)}\mathcal{W}^{\loc}_{\overline{g}}(v) \ \leq \ \liminf_{s \rightarrow 1^-}\left(\mathcal{W}^{\de, s}_{\overline{g_{\de, s}}}(\overline{u_{\de, s}}) - \min_{v \in H^{s, p, \de}_0(\Om_{-\de}; \R^n)}\mathcal{W}^{\de, s}_{\overline{g_{\de, s}}}(v)\right) \ \leq \ \ep,
\end{equation}
proving that $(\overline{u}, \overline{g}) \in \mathbb{T}^{\loc}_{\ep}$. 
     Now we prove that $(\overline{u}, \overline{g})$ is a solution to Problem \ref{Pr: LQuasiMinimizers} with tolerance $\ep$, so let $(u, g) \in \mathbb{T}^{\loc}_{\ep}$ be arbitrary, and let $w \in \argmin_{v \in W^{1, p}_0(\widetilde{\Om}; \R^n)}\mathcal{W}^{\loc}_g(v)$. We will construct a suitable, admissible recovery sequence. To this end, notice that by applying Proposition \ref{Prop: QuasiConvexityEnergy} for each $s < 1$, there exists a family $\{v_s\}_{s < 1}$ such that $v_s \in H^{s, p, \de}_0(\Om_{-\de}; \R^n)$ and $v_s \in \argmin_{v \in H^{s, p, \de}_0(\Om_{-\de}; \R^n)}\mathcal{W}^{\de, s}_g(v)$. 
     
     \textbf{Case 1:} Suppose that $\mathcal{W}^{\loc}_g(u) - \mathcal{W}^{\loc}_g(w) \ < \ \ep$. In that case, recall by \eqref{Gas->1Limeq} we have 
     \begin{equation}\label{Eq: QuasiMinimizersConvergenceOfMinss->1MinimizationEq1}
         \lim_{s \rightarrow 1^-}\mathcal{W}^{\de, s}_g(u) \ = \ \mathcal{W}^{\loc}_g(u).
     \end{equation}
     Notice that the energies $\{\mathcal{W}^{\de, s}_g(\cdot)\}_{s < 1}$ are equi-coercive, so  similarly to \eqref{Eq: QuasiMinimizersConvergenceOfMinss->1LimInfEq2} we have that
     \begin{equation}\label{Eq: QuasiMinimizersConvergenceOfMinss->1MinimizationEq2}
          \lim_{s \rightarrow 1^-}\mathcal{W}^{\de, s}_g(v_s) \ = \ \min_{v \in W^{1, p}_0(\Om; \R^n)}\mathcal{W}^{\loc}_g(v).
     \end{equation}
     Combining \eqref{Eq: QuasiMinimizersConvergenceOfMinss->1MinimizationEq1} and \eqref{Eq: QuasiMinimizersConvergenceOfMinss->1MinimizationEq2} and the Case 1 assumption, for all $s$ sufficiently close to $1$ we have 
     \begin{equation}\label{Eq: QuasiMinimizersConvergenceOfMinss->1MinimizationEq3}
         \mathcal{W}^{\de, s}_g(u) - \mathcal{W}^{\de, s}_g(v_s) \ < \ \ep.
     \end{equation}
     In other words, for all $s$ sufficiently close to $1$ we have $(u, g) \in \mathbb{T}^{\de, s}_{\ep}$, meaning in effect we may take the constant recovery sequence, i.e., $u_s := u$. 
     

     \textbf{Case 2:} Now suppose that $\mathcal{W}^{\loc}_g(u) - \mathcal{W}^{\loc}_g(w) = \ep$. Pick a sub-sequence $\{s(k)\}^{\infty}_{k = 1} \subset (0, 1)$ with $s_k \rightarrow 1$ as $k \rightarrow \infty$ and
     \begin{eqnarray}\label{Eq: QuasiMinimizersConvergenceOfMinss->1MinimizationEq4}
         \begin{aligned}
             |\mathcal{W}^{\loc}_g(u) - \mathcal{W}^{\de, s(k)}_g(u)| \ &< \ \frac{\ep}{2^{k + 2}}; \\
             |\mathcal{W}^{\loc}_g(w) - \mathcal{W}^{\de, s(k)}_g(v_{s(k)})| \ &< \ \frac{\ep}{2^{k + 2}}.
         \end{aligned}
     \end{eqnarray}
     Now we observe that for any fixed $s \in (0, 1)$ and any $T_s \in (\mathcal{W}^{\de, s}_g(v_s), \mathcal{W}^{\de, s}_g(u))$, there exists a $t_s \in (0, 1)$ so that $\mathcal{W}^{\de, s}_g(u_{t_s}) = T_s$, where $u_t := tu + (1 - t)v_s$. The reasoning is as follows: for $s \in (0, 1)$ fixed, the function $\mathcal{W}^{\de, s}_g(\cdot)$ is continuous with respect to strong convergence in $H^{s, p, \de}_0(\Om_{-\de}; \R^n)$ owing to the Generalized Dominated Convergence Theorem and the dominating growth bound \eqref{Eq: WBound} for $W$. Then we may use the Intermediate Value Theorem on the function $\mathcal{T}(t) := \mathcal{W}^{\de, s}_g(u_t)$ for any $T_s \in (\mathcal{W}^{\de, s}_g(v_s), \mathcal{W}^{\de, s}_g(u))$.
     

     Now we claim we may construct a sequence $\{t_{s(k)}\}^{\infty}_{k = 1} \subset [0, 1]$ such that $t_{s(k)} \rightarrow 1^-$ as $k \rightarrow \infty$, and
     \begin{equation}\label{Eq: cutoffExactValueCriterion}
         \mathcal{W}^{\de, s(k)}_g(u) - \mathcal{W}^{\de, s(k)}_g(u_{t_{s(k)}}) \ = \ \frac{\ep}{2^k}.
     \end{equation}
     We may apply the previous observation iteratively for each $k \in \N^+$ on the functionals $\{\mathcal{W}^{\de, s(k)}_g(\cdot)\}^{\infty}_{k = 1}$ to obtain a sequence $\{t_{s(k)}\}^{\infty}_{k = 1}$ satisfying \eqref{Eq: cutoffExactValueCriterion} (so, we are using the sub-sequence $\{s(k)\}^{\infty}_{k = 1}$ to dictate what $t_{s(k)}$ must be for each $k$). In fact, since $u_s$ converges to $u$ strongly in $L^p(\Om; \R^n)$ as $s \rightarrow 1^-$, we may take this sequence so that $t_{s(k)} \rightarrow 1^-$ as $k \rightarrow \infty$. We also remark that \eqref{Eq: cutoffExactValueCriterion} is feasible in the sense that $\mathcal{W}^{\loc}_g(u) - \mathcal{W}^{\loc}_g(w) = \ep$, so for $k$ sufficiently large we will be guaranteed that
     \begin{equation}\label{Eq: QuasiMinimizersConvergenceOfMinss->1MinimizationEq5}
         \mathcal{W}^{\de, s(k)}_g(u) - \mathcal{W}^{\de, s(k)}_g(v_{s(k)}) \ > \ \frac{\ep}{2}.
     \end{equation}
     In other words, the gap between $ \mathcal{W}^{\de, s(k)}_g(u)$ and $\mathcal{W}^{\de, s(k)}_g(v_{s(k)})$ is not so small as to make \eqref{Eq: cutoffExactValueCriterion} impossible to satisfy.

     Now we obtain the following inequality chain for all $k \in \N^+$ sufficiently large so that these inequalities are satisfied. Using \eqref{Eq: QuasiMinimizersConvergenceOfMinss->1MinimizationEq4} and \eqref{Eq: cutoffExactValueCriterion} we have:
     \begin{eqnarray}\label{Eq: grandInequalityChainEpsilon}
         \begin{aligned}
             &\mathcal{W}^{\de, s(k)}_g(u_{t_{s(k)}}) - \mathcal{W}^{\de, s(k)}_g(v_{s(k)}) \ = \ \\
             & \mathcal{W}^{\de, s(k)}_g(u) - \mathcal{W}^{\de, s(k)}_g(v_{s(k)}) - \frac{\ep}{2^k} \ < \ \\
             & \mathcal{W}^{\loc}_g(u) + \frac{\ep}{2^{k + 2}} - \mathcal{W}^{\de, s(k)}_g(v_{s(k)}) - \frac{\ep}{2^k} \ < \ \\
             & \mathcal{W}^{\loc}_g(u) + \frac{\ep}{2^{k + 2}} - \mathcal{W}^{\loc}_g(w) + \frac{\ep}{2^{k + 2}} - \frac{\ep}{2^k} \ = \ \\
             & \mathcal{W}^{\loc}_g(u) + \frac{\ep}{2^{k + 1}} - \mathcal{W}^{\loc}_g(w)  - \frac{\ep}{2^k} \ = \ \\
             & \mathcal{W}^{\loc}_g(u)  - \mathcal{W}^{\loc}_g(w)  - \frac{\ep}{2^{k + 1}} \ < \ \\
             &\mathcal{W}^{\loc}_g(u) - \mathcal{W}^{\loc}_g(w) \ = \ \ep
         \end{aligned}
     \end{eqnarray}
     In other words, we have that $(u_{t_{s(k)}}, g) \in \mathbb{T}^{\de, s(k)}_{\ep}$ for all $k \in \N^+$ sufficiently large, and $u_{t_{s(k)}} \rightarrow u$ strongly in $L^p(\Om; \R^n)$, making this a suitable recovery sequence.

Henceforth, for both Cases 1 and 2, we denote the recovery sequence simply by $\{(u_s, g)\}_{s < 1}$, for which $u_s \rightarrow u$ strongly in $L^p(\Om; \R^n)$.
By optimality we have that
\begin{equation}\label{Eq: QuasiMinimizersConvergenceOfMinss->1MinimizationEq9}
    \mathcal{F}_{\de, s}(\overline{u_{\de, s}}, \overline{g_{\de, s}}) \ \leq \ \mathcal{F}_{\de, s}(u_s, g).
    \end{equation}
    for all $s < 1$. We may send $s \rightarrow 1^-$. In doing so we use \eqref{Eq: lscFs->1} twice: once on the left-hand side of \eqref{Eq: QuasiMinimizersConvergenceOfMinss->1MinimizationEq9} with the family $\{\overline{u_s}\}_{s < 1}$ (alongside the weak convergence $\overline{g_{\de, s}} \rightharpoonup \overline{g}$ in $L^{p'}(\Om; \R^n)$), and once on the right-hand side of \eqref{Eq: QuasiMinimizersConvergenceOfMinss->1MinimizationEq9} with the family $\{u_s\}_{s < 1}$. This guarantees
    \begin{equation}\label{Eq: QuasiMinimizersConvergenceOfMinss->1MinimizationEq10}
        \mathcal{F}_{\loc}(\overline{u}, \overline{g}) \ \leq \ \mathcal{F}_{\loc}(u, g),
    \end{equation}
    which means that $(\overline{u}, \overline{g})$ indeed is a solution for Problem \ref{Pr: LQuasiMinimizers}. Finally, we may set $(u, g) := (\overline{u}, \overline{g})$ in the analysis used to prove optimality to deduce that the lim-sup inequality
    \begin{equation}\label{Eq: QuasiMinimizersConvergenceOfMinss->1Limsup}
        \limsup_{s \rightarrow 1^-}\mathcal{F}_{\de, s}(\overline{u_{\de, s}}, \overline{g_{\de, s}}) \ \leq \  \mathcal{F}_{\loc}(\overline{u}, \overline{g}).
    \end{equation}
    holds. Combining \eqref{Eq: QuasiMinimizersConvergenceOfMinss->1Limsup} with \eqref{Eq: QuasiMinimizersConvergenceOfMinss->1LimInfEq0} gives \eqref{Eq: QuasiMinimizersConvergenceOfMinss->1}, as desired.
 \end{proof}

 \begin{remark}\label{Rmk: Ep=0}
     The case where $\ep = 0$ corresponds to the situation studied in the \cite{cueto2026localization} with only partial success (the lim-inf inequality was proven for a more restricted class of cost functionals, but minimality of the resulting limit was not achieved). We note the argument used here clearly cannot be applied to that setting because of the prevalence of tolerance thresholds in \eqref{Eq: grandInequalityChainEpsilon}.
 \end{remark}

 \begin{remark}\label{Rmk: AltExistenceQuasiMins}
    Theorem \ref{Thm: QuasiMinimizersConvergenceOfMinss->1} provides an alternative mechanism for proving the existence of solutions to Problem \ref{Pr: LQuasiMinimizers}.
\end{remark}

\begin{corollary}[Strong convergence of controls as $s \rightarrow 1^-$]\label{Cor: QuasiMinimizersConvergenceOfMinss->1Controls}
    In the setting of Theorem \ref{Thm: QuasiMinimizersConvergenceOfMinss->1}, we additionally have that $\overline{g_{\de, s}} \rightarrow \overline{g}$ strongly in $L^r(\Om; \R^n)$ as $s \rightarrow 1^-$ for any $r \in [1, \infty)$.
\end{corollary}

This proof is similar to that of Theorem 4.12 in \cite{cueto2026localization} but holds for a broader class of energy densities.

\begin{proof}
We only show convergence of  $\overline{g_{\de, s}} \rightarrow \overline{g}$ strongly in $L^{p'}(\Om; \R^n)$, and then the extension to the other $L^r$ spaces proceeds exactly as in the proof of \cite[Theorem 4.12]{cueto2026localization}.

    Combining \eqref{Eq: QuasiMinimizersConvergenceOfMinss->1} and \eqref{Eq: QuasiMinimizersConvergenceOfMinss->1Controls1} yields the inequality
    \begin{equation}\label{Eq: QuasiMinimizersConvergenceOfMinss->1Controls2}
        \liminf_{s \rightarrow 1^-}\|\overline{g_{\de, s}}\|^{p'}_{L^{p'}_{\La}(\Om; \R^n)} \ \leq \ \|\overline{g}\|^{p'}_{L^{p'}_{\La}(\Om; \R^n)}
    \end{equation}
    Furthermore, \eqref{Eq: QuasiMinimizersConvergenceOfMinss->1Controls2} combined with the weak convergence $\overline{g_{\de, s}} \rightharpoonup \overline{g}$ in $L^{p'}(\Om; \R^n)$ gives the limit 
    \begin{equation}\label{Eq: QuasiMinimizersConvergenceOfMinss->1Controls3}
        \lim_{s \rightarrow 1^-}\|\overline{g_{\de, s}}\|^{p'}_{L^{p'}_{\La}(\Om; \R^n)} \ = \ \|\overline{g}\|^{p'}_{L^{p'}_{\La}(\Om; \R^n)}.
    \end{equation}
    Since $\La$ is a bounded positive weight function, $(L^{p'}(\Om; \R^n), \|\cdot\|_{L^{p'}_{\La}(\Om; \R^n)})$ is a reflexive Banach space. Then \eqref{Eq: QuasiMinimizersConvergenceOfMinss->1Controls3} and the weak convergence imply that $\overline{g_{\de, s}} \rightarrow \overline{g}$ strongly in $L^{p'}(\Om; \R^n)$, completing the proof.
\end{proof}

\begin{remark}[Relationship to control problem with global minimizers constraint]\label{Rmk: toGlobal}
    The nonlocal problem studied in \cite{cueto2026localization} is the same as Problem \ref{Pr: NLQuasiMinimizers}, except we replace the $\ep$-quasi-minimizer constraint with a global minimization constraint (i.e., $\ep = 0$). Thus it is natural to ask whether Problem \ref{Pr: NLQuasiMinimizers} and Theorem \ref{Thm: QuasiMinimizersConvergenceOfMinss->1} can be used to obtain convergence of minimizers as $s \rightarrow 1^-$ for the corresponding nonlocal problems with a global minimization constraint. 

    A logical first step is to prove that, if $\{(\overline{u_{s, \ep}}, \overline{g_{s, \ep}})\}_{\ep > 0}$ is a family of solutions to Problem \ref{Pr: NLQuasiMinimizers} (with $\de > 0$ and $s \in (0, 1]$ fixed), then there exists a pair $(\overline{u_s}, \overline{g_s})$ such that:
    \begin{itemize}
    \item $\overline{u_{s, \ep}} \rightarrow \overline{u_s}$ strongly in $L^p(\Om; \R^n)$;
    \item $\overline{g_{s, \ep}} \rightharpoonup \overline{g_s}$ weakly in $L^{p'}(\Om; \R^n)$.
        \item $(\overline{u_s}, \overline{g_s})$ solves the optimal control problem with $\ep = 0$.
    \end{itemize}
    The proof of this claim follows by the compactness result Proposition \ref{Prop: CompactEmbedding} and a direct verification of the optimality of $(\overline{u_s}, \overline{g_s})$ (i.e., prove that $\mathcal{F}_{\loc}(\overline{u_s}, \overline{g_s}) \leq \mathcal{F}_{\loc}(u_s, g_s)$ for every energy-minimizing pair $(u_s, g_s)$ using the optimality of $(\overline{u_{s, \ep}}, \overline{g_{s, \ep}})$ for each $\ep > 0$). 

    However, this is insufficient for proving convergence of solutions to the problems with the global minimization energy constraint because, for any $s \in (0, 1]$ and optimal pair $(\overline{u_s}, \overline{g_s})$, we cannot guarantee a priori the existence of optimal pairs $\{(\overline{u_{s, \ep}}, \overline{g_{s, \ep}})\}_{\ep > 0}$ that convergence (in a suitable topology) to $(\overline{u_s}, \overline{g_s})$ as $\ep \rightarrow 0^+$. The main reason for this is that the problems where $\ep = 0$ do not necessarily have unique solutions.
\end{remark}


 \subsection{Convergence as $\de \rightarrow 0^+$}\label{Subsec: QuasiMinsConvDe->0}

 Now we discuss the convergence results where $s \in (0, 1)$ is fixed and $\de \rightarrow 0^+$; the progression here is identical to that of Subsection \ref{Subsec: quasiMinimizersAsymptoticss->1}. This time we need a different auxiliary assumption on the kernels $\{w_{\de}(\cdot)\}_{\de > 0}$.

\begin{assumption}\label{Assump: kernelde->0}
    Let $w(x):=w_1(x)$ satisfy Assumption \ref{Assump: Kernels}. Then, the sequence of $w_\de$ are constructed via the formula $w_\de(x):=w\left(\frac{x}{\de}\right)$. 
\end{assumption}

This assumption is required for a Nonlocal Poincaré Inequality with constant independent of $\de$ (see \cite[Corollary 3.4]{cueto2025gamma}). 

    \begin{proposition}[Nonlocal Poincaré inequality II]\label{Prop: NLPoincarede}
		Let $\de \geq 0$, $p \in (1, \infty)$, $\Om \subset \R^n$ be open and bounded. Then there exists a constant $C > 0$ depending only on $\Om$, $s$, and $p$ such that for all $u \in H^{s, p, \de}_0(\Om_{-\de}; \R^n)$, we have
		\begin{equation}\label{Eq: NLPoincarede}
			\|u\|_{L^p(\Om; \R^n)} \ \leq \ C\|D^s_{\de}u\|_{L^p(\Om; \R^{n\times n})}.
		\end{equation}
	\end{proposition}
    
Next we state the $\Ga$-convergence result for the energies.

 \begin{theorem}[$\Ga$-convergence as $\de \rightarrow 0^+$]\label{Thm: Gade->0: GaConv}
	Let $s \in (0, 1)$, $g \in Z_{\text{ad}}$ be fixed. Then the family of functionals $\{\mathcal{W}^{\de, s}_g\}_{s < 1}$ will $\Ga$-converge in the strong $L^p(\Om; \R^n)$ topology to $\mathcal{W}^{\loc}_g$, which we will denote $\mathcal{W}^{\de, s}_g \xrightarrow{\Ga, \ \de \rightarrow 0^+} \mathcal{W}^{\loc}_g$. In other words, we have the following:
	\begin{enumerate}[label=\textbf{GCs\arabic*}]
		\item\label{GCd1} If $\{u_{\de}\}_{\de > 0} \subset L^p(\Om; \R^n)$ is a sequence such that $u_{\de} \rightarrow u$ strongly in $L^p(\Om; \R^n)$, then we have the \textbf{lim-inf inequality}
		\begin{equation}\label{Gade->0Liminf}
			\mathcal{W}^{\loc}_g(u) \ \leq \ \liminf_{\de \rightarrow 0^+}\mathcal{W}^{\de, s}_g(u_{\de}).
		\end{equation}
		\item\label{GCd2} If $u \in L^p(\Om; \R^n)$, then there exists a \textbf{recovery sequence} of $\{u_{\de}\}_{\de > 0} \subset L^p(\Om; \R^n)$ such that $u_{\de} \rightarrow u$ strongly in $L^p(\Om; \R^n)$ and
		\begin{equation}\label{Gade->0Limsup}
			\mathcal{W}^{\loc}_g(u) \ \geq \ \limsup_{\de \rightarrow 0^+} \mathcal{W}^{\de, s}_g(u_{\de}).
		\end{equation}
	\end{enumerate}
\end{theorem}

This precise result is Theorem 4.15 in \cite{cueto2026localization}, but that in turn is an application of \cite[Theorem 3.7]{cueto2025gamma}, which describes the analysis of localizing energy functionals in a vanishing horizon parameter.

To prove our main convergence result, we also need a version of Lemma \ref{Lem: boundednessAdmissibleStatesNLQuasi} for varying $\de$. 

 \begin{lemma}[Boundedness of admissible states as $\de \rightarrow 0^+$]\label{Lem: boundednessAdmissibleStatesNLQuaside->0}
        Let $W: \Om \times \R^n \times \R^{n \times n} \rightarrow \R$ be an energy density satisfying Assumption \ref{Assump: energyDensityAssump}, and let $\ep > 0$, $s \in (0, 1)$ be fixed. Let $\{(u_{\de}, g_{\de})\}_{\de > 0}$ be a family of pairs such that $(u_{\de}, g_{\de}) \in \mathbb{T}^{\de, s}_{\ep}$ for all $\de > 0$. Then $\liminf_{\de \rightarrow 0^+}\|D^s_{\de}u_{\de}\|_{L^p(\Om_{\de}; \R^{n \times n})} < \infty$.
    \end{lemma}
    
Finally, we need a version of Lemma \ref{Lem: lscs->1} for when $\de$ approaches $0$.

\begin{lemma}\label{Lem: lscde->0}
     If $s \in (0, 1)$ is fixed and $\{u_{\de}\}_{\de > 0}$ is a family such that $\liminf_{\de \rightarrow 0^+}\|u_{\de}\|_{H^{s, p, \de}(\Om; \R^n)} < \infty$, $u_{\de} \rightarrow u$ strongly in $L^p(\Om; \R^n)$, and $D^s_{\de} u_{\de} \rightharpoonup \grad u$ weakly in $L^p(\Om; \R^{n \times n})$, then
            \begin{equation}\label{Eq: lscde->0}
                \liminf_{\de \rightarrow 0^+}\int_{\Om}F(x, u_{\de}(x), D^s_{\de}u_{\de}(x))dx \ \geq \ \int_{\Om}F(x, u(x), \grad u(x))dx.
            \end{equation}
    
\end{lemma}

\begin{proof}
    The proof coincides with the lim-inf inequality portion of the proof of \cite[Theorem 3.7]{cueto2025gamma}. The differences are twofold: one, we integrate over $\Om$ instead of $\Om_{\de}$; and second, our integrands exhibit $u$-dependence. The latter is not a complication for the proof since the weak lower semi-continuity result \cite[Theorem 8.11]{dacorogna2007direct} applies to functionals with $u$-dependence. 
\end{proof}

Now we state the main convergence result as $\de \rightarrow 0^+$, analogous to Theorem \ref{Thm: QuasiMinimizersConvergenceOfMinss->1}. However, we need to use \cite[Lemma 3.6]{cueto2025gamma} as our compactness lemma for vanishing horizons.

\begin{theorem}[Convergence of minimizers with $\de \rightarrow 0^+$]\label{Thm: QuasiMinimizersConvergenceOfMinsDe->0}
    Suppose that $(\overline{u_{\de, s}}, \overline{g_{\de, s}}) \in \mathbb{T}^{\de, s}_{\ep}$ is a solution to Problem \ref{Pr: NLQuasiMinimizers} with energy tolerance $\ep$. Then there exists a non-relabeled sub-sequence, a $\overline{u} \in W^{1, p}_0(\Om; \R^n)$, and a $\overline{g} \in Z_{\text{ad}}$ such that $(\overline{u}, \overline{g}) \in \mathbb{T}^{\loc}_{\ep}$, we have
    \begin{equation}\label{Eq: QuasiMinimizersConvergenceOfMinsDe->0}
        \lim_{\de \rightarrow 0^+}\mathcal{F}_{\de, s}(\overline{u_{\de, s}}, \overline{g_{\de, s}}) \ = \ \mathcal{F}_{\loc}(\overline{u}, \overline{g}),
    \end{equation}
    and that $(\overline{u}, \overline{g})$ is a solution to Problem \ref{Pr: LQuasiMinimizers} with energy tolerance $\ep$. Moreover, $\overline{u_{\de, s}} \rightarrow \overline{u}$ strongly in $L^p(\Om; \R^n)$ and $\overline{g_{\de, s}} \rightharpoonup \overline{g}$ weakly in $L^{p'}(\Om; \R^n)$.
 \end{theorem}

 Finally, we provide a version of Corollary \ref{Cor: QuasiMinimizersConvergenceOfMinss->1Controls} when $\de \rightarrow 0^+$.

 \begin{corollary}[Strong convergence of controls as $\de \rightarrow 0^+$]\label{Cor: QuasiMinimizersConvergenceOfMinsde->0Controls}
    In the setting of Theorem \ref{Thm: QuasiMinimizersConvergenceOfMinsDe->0}, we additionally have that $\overline{g_{\de, s}} \rightarrow \overline{g}$ strongly in $L^r(\Om; \R^n)$ as $\de \rightarrow 0^+$ for any $r \in [1, \infty)$.
\end{corollary}



\section{Concluding Remarks}\label{Sec: Conclusion}

In summary, this paper provided a nonlocal optimal control formulation where the constraint involved approximating minimizers of a quasi-convex energy through the notion of $\ep$-quasi-minimizers. We proved existence results for this class of control problems, and demonstrated that this formulation allows for convergence of solutions in the vanishing nonlocal limit or in localization of the fractional parameter, for a general range of cost functionals. Our results are not particularly reliant on the function space framework chosen and the techniques used are readily applicable to other families of parameterized optimal control problems, particularly ill-posed problems with at least one solution.

 One option for exploration in this direction is to enforce admissible pairs of controls and states to be equilibrium points of the energy rather than just focus on minimizers. If the forcing terms have high-order nonlinearities then multiplicity of equilibrium points for any fixed force may be guaranteed, following the theory of \cite{torres2017existence, servadei2012mountain, servadei2013variational}; \cite{siktar2024superlinear} introduces some relevant techniques for these control problems, albeit in the fractional diffusion setting.

We also highlight the recent paper \cite{grekas2025convergence} for its theory and implementation of a Discontinuous Galerkin finite element method for approximating minimizers of quasi-convex energies. One could attempt to extend that theory to optimal control problems; this analysis is open even for PDE-based control problems.

\bibliographystyle{plain} 
\bibliography{refs} 

\end{document}